\documentclass[12pt]{amsart}

\usepackage{amsmath,amssymb, amscd, stmaryrd}
\usepackage{fullpage}
\usepackage{enumerate}
\usepackage{cite}
\usepackage{framed}
\usepackage[dvips,dvipdf]{graphicx}
\usepackage[usenames,dvipsnames]{color}
\usepackage{color}
\usepackage{tabularx}
\usepackage{array}
\usepackage{multirow}

\usepackage{amsfonts}
\usepackage{epstopdf}
\usepackage{subfigure,caption,float}
\usepackage{enumerate}
\usepackage{cleveref}

\newcommand{\ncom}{\newcommand}
\ncom{\beqn}{\begin{eqnarray*}}
	\ncom{\eeqn}{\end{eqnarray*}}

\newtheorem{theorem}{Theorem}[section]

\newtheorem{lemma}[theorem]{Lemma}

\theoremstyle{definition}

\theoremstyle{definition}

\theoremstyle{remark}
\newtheorem{remark}[theorem]{Remark}
\newtheorem{example}[theorem]{Example}

\title{A Fourier extension based numerical integration scheme for fast and high-order approximation of convolutions with weakly singular kernels }

\author[A. Anand]{Akash Anand} 
\address{Akash Anand, Department of
	Mathematics and Statistics, Indian Institute of Technology, Kanpur, UP 208016}
\email{akasha@iitk.ac.in}

\author[A. K. Tiwari]{Awanish Kumar Tiwari} 
\address{Awanish Kumar Tiwari, Department of
	Mathematics and Statistics, Indian Institute of Technology, Kanpur, UP 208016}
\email{awanish@iitk.ac.in}

\begin{document}
\maketitle{\title}
\begin{abstract}
   Computationally efficient numerical methods for high-order approximations of convolution integrals involving weakly singular kernels find many practical applications including those in the development of fast quadrature methods for numerical solution of integral equations. Most fast techniques in this direction utilize uniform grid discretizations of the integral that facilitate the use of FFT for $O(n\log n)$ computations on a grid of size $n$. In general, however, the resulting error converges slowly with increasing $n$ when the integrand does not have a smooth periodic extension. Such extensions, in fact, are often discontinuous and, therefore, their approximations by truncated Fourier series suffer from Gibb's oscillations. In this paper, we present and analyze an $O(n\log n)$ scheme, based on a Fourier extension approach for removing such unwanted oscillations, that not only converges with high-order but is also relatively simple to implement. We include a theoretical error analysis as well as a wide variety of numerical experiments to demonstrate its efficacy.
\end{abstract}

\section{Introduction}
\label{sec:introduction}

In this paper, we consider the problem of approximating the integral operator $A : C([0,1]) \to C([0,1])$ given by
\begin{equation}\label{eq:-Ie}
(Au)(x) = \displaystyle \int_{0}^{1}g(x-y)u(y)dy
\end{equation}
with a weekly singular kernel $g$
using a numerical integration scheme of the form
\begin{align}
\label{gen_quad}
(A_n u)(x) = \sum_{j=0}^{n} w_j^n(x)u(x_j^n)
\end{align}
with quadrature points $x_j^n = j/n$ and weights $w_j^n(x)$ that depend on the kernel $g$. While our discussions in this text remain valid for general weakly singular kernels, for concreteness of convergence rates, we investigate $g$ either of the form $g(x) = |x|^{\gamma}, \gamma \in (-1,\infty)$ or of the form $g(x) = \log |x|$.
The primary aim of this effort is to develop a quadrature so that the approximations in \cref{gen_quad} ---
\begin{enumerate}[a.]
	\item
	converge with high-order, that is, satisfy
	\[
	\Vert A_n u - A u\Vert_{\infty} \le C(u)n^{-\mu(r)}, \ \ \ u \in C^{r}([0,1]),
	\]
	for some increasing and positive function $\mu : \mathbb{N} \to \mathbb{R}^+$, and
	\item
	can be evaluated at quadrature points efficiently, that is, the set $\{ (A_n u)(x_j^n) : j = 0, \ldots, n \}$ can be computed in $O(n\log n)$ operations.
\end{enumerate}
For example, such a quadrature can then be utilized toward obtaining an $O(n\log n)$ Nystr\"{o}m solver for numerical solution of integral equations of the form
\[
u(x) - \int_0^1 g(x-y)m(y)u(y)\,dy = f(x), \ \ x \in [0,1]
\]
for $f, m \in C^{r}([0,1])$.

While several numerical integration schemes for integrands that have a point-singularity within the domain of integration have been developed, this problem remains an active area of research. 
	In the special case when the integrand singularity is at one or both endpoints of the interval, an appropriate change of variable, as suggested in \cite{Kress}, can be employed for their analytical resolution were a subsequent use of trapezoidal rule yields a high-order numerical accuracy.
	%
The more general case where the singularity is an interior point of the interval requires more careful treatment. 
Among the available techniques, a prominent example is a high-order corrected trapezoidal rule due to Rokhlin \cite{Rokhlin1990}. While a rapid increase in the magnitude of correction weights with increasing order limits their original approach to low order convergence rates, Kapur and Rokhlin subsequently introduced another scheme (see \cite{Rokhlin1997}), that handles the singularity by separating the integrand into regular and singular parts along with allowing some quadrature nodes to lie outside the interval of integration. 
In this case, however, the linear system for determining the nodes and weights, 
especially near the interval endpoints, are poorly conditioned that worsen rapidly with increasing order. 

In a subsequent effort, Alpert \cite{Alpert} introduced a hybrid Gauss-Trapezoidal rule that is based on the periodic trapezoidal rule. 
In this scheme, he replaced the equispaced nodes near each end  with an optimal set of auxiliary nodes that are chosen using the algorithm of Kolm--Rokhlin\cite{Kolm}.
 Later, based on a combination of boundary and singularity correction, Aguilar and Chen proposed a corrected trapezoidal rule for logarithmic singularity in $\mathbb{R}^2$ \cite{Augilar2002} 
 and for the singularity of the form $\displaystyle 1/|x|$ in $\mathbb{R}^3$ \cite{Augilar2005}. 
 While they numerically demonstrated the effectiveness of this scheme in computing singular integrals, a corresponding convergence analysis is not readily available. 

More recently, Duan and Rokhlin \cite{Duon} introduced a class of quadrature formulae for handling singularities of the form $\log|x|$ in two dimensions and $1/|x|$ in three dimensions. 
	This approach is somewhat related to the Ewald summation \cite{Ewald} 
and leads to quadratures that can be viewed as a version of the corrected trapezoidal rule.  
To avoid solving ill-conditioned linear systems for obtaining the quadrature weights, the scheme uses very lengthy analytic expressions that depend on the underlying mesh.
A recent paper in this direction  by Marin \cite{Marin}, is again a corrected trapezoidal approach,  
where singular functions of the form $|x|^{\gamma}$,  $\gamma>-1$ in one dimension 
and $|x|^{-1}$ in two dimensions were considered for high-order numerical integration.

The scheme introduced in this paper utilizes a regular grid in order to employ the fast Fourier transform (FFT) to obtain convolution values in $O(n \log n)$ computational time. Indeed, as
\begin{align}\label{eq:-fa}
(Au)(x) = 
\sum_{k=-\infty}^{\infty} \widehat{g_e}(k)\widehat{u_e}(k) e^{2\pi i k x},
\end{align}
where 
$g_e$ and $u_e$ denotes the periodic extension of $g$ and $u$ respectively and 
for $1$-periodic function $v$, $\hat{v}$ denotes its Fourier coefficients given by
\[
\hat{v}(k) = \int_0^1 v(x)e^{-2\pi i k x}\,dx,
\]
the inverse FFT evaluates the truncated series 
\[
(A_nu)(x_j^n) = \sum_{k=-n/2}^{n/2-1} \widehat{g_e}(k)\widehat{u_e}(k) e^{2\pi i k x_j^n}
\]
at $x_j^n, j = 0, \ldots, n-1$ in a straightforward manner in $O(n \log n)$ operations. However, in general, the resuting error
\[
\Vert Au -A_nu \Vert_{\infty} = \sum_{|k| \ge n/2} |\widehat{g_e}(k)||\widehat{u_e}(k)|
\]
 converges slowly with increasing $n$ owing to discontinuous $u_e$ and singular $g_e$. In this paper, we suggest an approach to alleviate the difficulty of slow convergence by making $u_e$ more amenable to Fourier approximation that also allows us to pre-compute the Fourier coefficients for $g_e$ more accurately resulting in an overall rapidly convergent numerical integration scheme.

The organization of the paper is as follows. In \cref{sec:-fc}, we describe the construction of Fourier extension for smooth functions 
where the functional data is available on equispaced grids. Next, in \cref{sec:-IntSc}, we provide the proposed numerical integration scheme for the approximation of the convolution $A$. We then present a theoretical analysis of our integration scheme in \cref{sec:-qa}. In \cref{sec:-Bk}, we present the numerical computation of singular moments that occur in the numerical integration scheme. Further, as a special case, we discuss our numerical integration scheme for compactly supported function in \cref{sec:-csi}. A variety of numerical results to validate the accuracy of our quadrature implemented in this paper are presented in \cref{sec:-ne}.

\section{Fourier extension}\label{sec:-fc}

We consider the problem of approximation of a function $u \in C^{\infty}([0,1])$ by a truncated Fourier series where discrete functional data is available on an equispaced grid on the interval $[0,1]$. Recall that the Fourier series approximations fail to converge uniformly when $u(0)\ne u(1)$ due to
rapid oscillations near the boundary, known as the Gibb’s phenomenon \cite{Gibbs1,Gibbs2,Hewitt,Wilbraham}. Several approximation approaches have been proposed to overcome the difficulty of Gibb’s oscillations. These  include  schemes  that  utilize  Fourier  or  physical  space  filters  \cite{Shu5}  as  well  as  those  that  project  the partial  Fourier  sums  onto  suitable  functional  spaces.   For  example,  the  Gegenbauer  projection  technique \cite{Shu1,Shu2,Shu3,Shu4,Shu5, Shu6} utilizes a space spanned by Gegenbauer polynomials.  In Fourier-pad\'{e} approximations, partial Fourier sums are approximated by rational trigonometric functions \cite{Fornberg
,Geer,Banerjee}. Techniques based on extrapolation algorithms \cite{Brezinski} have also been used.  Several Fourier extension ideas have also been proposed that seek to find a trigonometric polynomial of the form
\[ \sum_{k=-n}^{n-1} \widehat{u_{c,e}}(k) e^{2\pi i k x/b}\] with $b\ge 1$, where $u_{c,e}$ is the $b$--periodic extension of $u_c$, the continuation of $u$ on $[0,1]$ to $u_c$ on $[0,b]$ or $[1-b,1]$ in such a way that $u_c\equiv u$ on $[0,1]$ and $u^{(\ell)}(0) =f^{(\ell)}(b)$ for all integers $0\le \ell \le r,$ for some $r >0$.  Once such an $u_c$ has been produced, the restriction of its truncated Fourier series to $[0,1]$ serves as an approximation to $u$.  For some examples where Fourier extension ideas have been used and discussed in various contexts, see \cite{Israeli,Boyd,Garbey,Dervout,Huybrechs, Nieslony}.

In the present context,  
 a grid of size $n+1$ has its $j$th grid point at $x_j = j/n,$ where the corresponding function value $u(x_j)$ is assumed to be known and are denoted by $u_j$ for $j = 0, \ldots, n$.
The function $u$ is continued to the interval $[-1,1]$ as
\begin{equation}
\label{eq:fc_discrete}
u_c(x) = 
\begin{cases}
u(x), & x \in [0,1] \\
p(U)(x), & x \in [-1,0),
\end{cases}
\end{equation}
where, for a $2 \times (r+1)$ matrix $U = (u_{ij})_{0 \le i \le 1, 0 \le j \le r}$, $p(U)$ is a polynomial of degree $2r+1$ of the form
\begin{align}
\label{poly}
p(U)(x) &= \sum_{m=0}^r u_{0m}\ p_m^0(x) + \sum_{m=0}^r u_{1m}\ p_m^1(x)
\end{align}
with
\[
p_m^0(x) = \frac{1}{m!} x^{m}(1+x)^{r+1}  \sum_{n = 0}^{r-m} (-x)^{n} \binom{r+n}{n}
\]
and
\[
p_m^1(x) = \frac{1}{m!}  (1+x)^{m}(-x)^{r+1}  \sum_{n = 0}^{r-m} (1+x)^{n} \binom{r+n}{n}.
\]
It is easy to check that, for $0 \le m \le r$, the $m$-th derivative of the polynomial $p(U)$ satisfies
\begin{align*}
p(U)^{(m)}(0) = u_{0m}, \ \ \ p(U)^{(m)}(-1) = u_{1m}.
\end{align*}
Thus, choosing the entries of $U$ as
\[
u_{0m} = u^{(m)}(0) \ \ \text{ and }\ \ u_{1m} = u^{(m)}(1)
\]
ensures that $u_c \in C^r([-1,1])$. The continued discrete data $(u_{c})_j = u_c(j/n), j = -n, \ldots, n-1$ is then used to obtain
\begin{align} \label{fc_coeffs_e}
\hat{u}_{c,n}(k) = 
\frac{1}{2n}\sum_{j=-n}^{n-1} \left(u_c\right)_{j} e^{-\pi i j k /n}
\end{align}
for $k = -n,\ldots, n-1$,
to arrive at the interpolating Fourier extension approximation for the discrete problem given by
\begin{align} \label{tfse}
u_{c,n}(x) = \sum_{k=-n}^{n-1} \hat{u}_{c,n}(k) e^{\pi i k x}.
\end{align}

However, in general, the derivative data $u^{(m)}(0)$ and $u^{(m)}(1)$ may not be available exactly and their approximations need to be obtained using the discrete data $u_j$. In such a scenario, we replace $U$ by its approximation $U^q$ with entries obtained as $u_{00}^q = u(x_0)$, $u_{10}^q = u(x_n)$, and for $1 \le m \le r$, we set
\[
u^q_{0m} = \mathcal{D}_{n,q}^{m,+}(u)(x_0) \text{ and } u^q_{1m} = \mathcal{D}_{n,q}^{m,-}(u)(x_n),
\]
using forward and backward finite difference derivative operators $\mathcal{D}_{n,q}^{m,+}(u)$ and $\mathcal{D}_{n,q}^{m,-}(u)$ respectively of the order of accuracy $q$ as approximations of $u^{(m)}$ whose generic form reads
\[
D_{n,q}^{m,\pm}(u)(x_\ell) = (\pm n)^m\left( \sum_{k=0}^{m+q-1} \left(a_{q}^m\right)_k u_{\ell\pm k} \right)
\]
for appropriately chosen constants $\left(a_{q}^m\right)_k$. The continuation of $u$ corresponding to the boundary data matrix $U^q$ is denoted by $u_c^q$.
The continued data $(u_{c}^q)_j = u_c^q(j/n), j = -n, \ldots, n-1$ is then used to similarly obtain
\begin{align} \label{fc_coeffs_a}
\hat{u}_{c,n}^q(k) = 
\frac{1}{2n}\sum_{j=-n}^{n-1} \left(u_c^q\right)_{j} e^{-\pi i j k /n}, 
\end{align}
for $k = -n,\ldots, n-1$,
and the corresponding Fourier extension approximation given by
\begin{align} \label{tfsa}
u_{c,n}^q(x) = \sum_{k=-n}^{n-1} \hat{u}_{c,n}^q(k) e^{\pi i k x}.
\end{align}
Note that the coefficients $\hat{u}_{c,n}^q(k)$ can be computed in $\mathcal{O}(n\log n)$ computational time using the FFT.
We use the approximation $u_{c,n}^q$ of $u$ to construct a numerical integration scheme $A_n$, that we discuss next.


\section{The Integration scheme}\label{sec:-IntSc}

We begin by observing that
\[
(Au)(x) = \int_0^1 g(x-y)u(y)\,dy = \int_0^1 g(x-y)u_c(y)\,dy
\]
where $u_c$ is the continuation of $u$ to $[-1,1]$ as described in the previous section. We rewrite $Au$ as
\begin{align*}
          (Au)(x) 
               & = \int_{x-1}^{x+1} g(x-y)u_c(y)dy - (C_LU)(x) - (C_RU)(x)
\end{align*}                               
where 
$C_L$ and $C_R$ are given by
\begin{align*}
(C_LU)(x) &= \int_{x-1}^{0} g(x-y)p(U)(y)dy, \\
(C_RU)(x) &= \int_{1}^{x+1} g(x-y)p(U)(y-2)dy,
\end{align*}

The Fourier series for the $2$-periodic extension $u_{c,e}$ of $u_c$ can then be utilized to write the integral as
\begin{align*}
 (Au)(x) 
          &= \sum_{k=-\infty}^{\infty} \beta(k) \widehat{u_{c,e}}(k)  e^{\pi ikx} -(C_LU)(x)-(C_RU)(x),
\end{align*}
where 
\[
\widehat{u_{c,e}}(k) = \frac{1}{2} \int_{x-1}^{x+1} u_c(y) e^{-\pi iky}\,dy,
\]
and
\begin{align}
\label{eq:beta}
\beta(k) 
= \int_{-1}^{1} g(\rho)  e^{\pi ik\rho}\,d\rho.
\end{align}
In view of this, we take the numerical integration scheme $A_n$ to be of the form
\begin{align}
\label{eq:conv}
(A_nu)(x) = \sum_{k=-n}^{n-1} \beta(k) \hat{u}_{c,n}^q(k)  e^{ \pi ikx} -(C_LU^q)(x)-(C_RU^q)(x)
\end{align}
which, further, can we rewritten in the form of a quadrature on the equispaced grid given by
\begin{align}
(A_nu)(x) 
&= \sum_{j=0}^{n-1} w_j^n(x) u_{j}  + (C_nU^q)(x)
\label{eq:quad}
\end{align}
with the correction term
\begin{align*}
(C_nU^q)(x) = \sum_{j=-n}^{-1} w_j^np(U^q)(j/n) 
- (C_LU^q)(x) - (C_RU^q)(x)
\end{align*}
and quadrature weights
\begin{align*}
w_j^n(x) = \frac{1}{2n}\sum_{k=-n}^{n-1} \beta(k)e^{\pi i k (x-j/n)}.
\end{align*}
Note that while \cref{eq:conv} provides an $O(n\log n)$ scheme for obtaining the convolution on an equispaced grid on size $n$, one could use \cref{eq:quad} with the pre-computed weights and the correction term for a single point integration. Use of either form requires the values $\beta(k)$ that, in principle, can be obtained exactly through the symbolic integration of \cref{eq:beta}. In practice, however, it might be more convenient to numerically pre-compute them to high precision and store them for later calculations. We discuss a couple of numerical strategies toward this in \cref{sec:-Bk}.

The functions $(C_LU)(x)$ and $(C_R)(x)$, the left and right corrections in \cref{eq:conv}, can be obtained from analytical expressions that can be easily derived using the form of continuation polynomial. For example, if $g(x) = |x|^\gamma$, then 
\begin{align*}
&(C_LU)(x) =  \int_{x-1}^{0} g(x-y)p(U)(y)\,dy = \\
     & \sum_{m=0}^r \frac{u_{0m}}{m!} \sum_{n=0}^{r-m} (-1)^n \binom{r+n}{n} \sum_{q=0}^{r+1} \binom{r+1}{q} \sum_{k=0}^{m+n+q} (-1)^k \binom{m+n+q}{k} x^{m+n+q-k} \left(\frac{1-x^{k+1+\gamma}}{k+1+\gamma}\right)+\\ 
     &(-1)^{r+1} \sum_{m=0}^r \frac{u_{1m}}{m!} \sum_{n=0}^{r-m} \binom{r+n}{n} \sum_{q=0}^{m+n} \binom{m+n}{q} \sum_{k=0}^{r+q+1} (-1)^k \binom{r+q+1}{k} x^{r+q+1-k} \left(\frac{1-x^{k+1+\gamma}}{k+1+\gamma}\right),      
\end{align*}
and
\begin{align*}
&(C_RU)(x) =  \int_{1}^{x+1} g(x-y)p(U)(y-2)\,dy = \\
     &\sum_{m=0}^r \frac{u_{0m}}{m!} \sum_{n=0}^{r-m} (-1)^n \binom{r+n}{n} \sum_{q=0}^{r+1} \binom{r+1}{q} \sum_{k=0}^{m+n+q} \binom{m+n+q}{k} (x-2)^{m+n+q-k} \left(\frac{1-(1-x)^{k+1+\gamma}}{k+1+\gamma}\right) +\\ 
     & (-1)^{r+1} \sum_{m=0}^r \frac{u_{1m}}{m!} \sum_{n=0}^{r-m} \binom{r+n}{n} \sum_{q=0}^{m+n} \binom{m+n}{q} \sum_{k=0}^{r+q+1} \binom{r+q+1}{k} (x-2)^{r+q+1-k} \left(\frac{1-(1-x)^{k+1+\gamma}}{k+1+\gamma}\right).        
\end{align*}
Similar expressions can also be obtained in a straightforward manner for $g(x) = \log(|x|)$. In the next section, we investigate the errors that are associated with the proposed numerical scheme.

 \section{Error Analysis} \label{sec:-qa}
\label{sec:Quadrature Analysis}


Clearly, the error in the numerical scheme is given by
\begin{align} \nonumber
(Au)(x) - (A_nu)(x) &= 
\sum_{k=-n}^{n-1} \beta(k) \left(\widehat{u_{c,e}}(k) - \hat{u}_{c,n}^q(k)\right)  e^{i \pi kx} + \sum_{\substack{k\ge n \\ k < -n}} \beta(k) \widehat{u_{c,e}}(k)  e^{\pi ikx} \\ \nonumber
&- (C_LU)(x) - (C_RU)(x) + (C_LU^q)(x) + (C_RU^q)(x) \\ 
&= \sum_{k=-n}^{n-1} \beta(k) \left(\widehat{u_{c,e}}(k) - \hat{u}_{c,n}(k)\right)e^{\pi ikx}   + \sum_{k=-n}^{n-1} \beta(k) \left(\hat{u}_{c,n}(k) - \hat{u}_{c,n}^q(k)\right)e^{\pi ikx}  \label{eq:na} \\ 
& + \sum_{\substack{k \ge n \\ k < -n}} \beta(k) \widehat{u_{c,e}}(k)e^{\pi ikx} - (C_L(U-U^q))(x) - (C_R(U-U^q))(x) \nonumber
\end{align}
To estimate the magnitude of $$ \sum_{k=-n}^{n-1} \beta(k) \left(\widehat{u_{c,e}}(k) - \hat{u}_{c,n}(k)\right)e^{\pi ikx} + \sum_{\substack{k \ge n \\ k < -n}} \beta(k) \widehat{u_{c,e}}(k)e^{\pi ikx},$$ we recall that $\beta(k) = \beta(-k)$ and
\begin{align*}
\hat{u}_{c,n}(k) - \widehat{u_{c,e}}(k) &=  \frac{1}{2n}\sum_{j=-n}^{n-1} (u_c)_{j} e^{-\pi i j k /n} - \widehat{u_{c,e}}(k)
 = \sum_{\ell=1}^{\infty} \left(\widehat{u_{c,e}}(k+2\ell n) + \widehat{u_{c,e}}(k-2\ell n)\right)
\end{align*}
to arrive at the upper bound
\begin{align}
&\sum_{k=-n}^{n-1} |\beta(k)| \sum_{\substack{\ell=1}}^{\infty} |\widehat{u_{c,e}}(k+2\ell n) + \widehat{u_{c,e}}(-k-2\ell n)| + \sum_{|k| \ge n}|\beta(k)||\widehat{u_{c,e}}(k)+\widehat{u_{c,e}}(-k)| \label{eq:err1}.
\end{align}
We also observe that, for $x \in [0,1]$, we have
\begin{align*}
(C_L(U-U^q))(x) + (C_R(U-U^q))(x) = \sum_{\ell=-\infty}^\infty \beta(k)\widehat{(u_{c,e}-u_{c,e}^q)}(k)e^{\pi i k x},
\end{align*}
and, therefore, we get
\begin{align*}
&(C_L(U-U^q))(x) + (C_R(U-U^q))(x)-\sum_{k=-n}^{n-1}\beta(k)\left( \hat{u}_{c,n}(k) - \hat{u}_{c,n}^q(k)\right) = \\
&\sum_{k=-\infty}^\infty \beta(k)\widehat{(u_{c,e}-u_{c,e}^q)}(k)e^{\pi i k x} - \frac{1}{2n} \sum_{k=-n}^{n-1} \beta(k)  \sum_{j=-n}^{n-1} \left((u_c)_j-(u_c^q)_{j}\right) e^{-\pi i j k /n} = \\
&\sum_{k=-\infty}^\infty \beta(k)\widehat{(u_{c,e}-u_{c,e}^q)}(k)e^{\pi i k x} - \sum_{k=-n}^{n-1} \beta(k) \sum_{\ell=-\infty}^{\infty}   \widehat{(u_{c,e}-u_{c,e}^q)}(k+2\ell n) = \\
&-\sum_{k=-n}^{n-1}\beta(k)\sum_{\substack{\ell=-\infty \\ \ell \ne 0}}^{\infty}    \widehat{(u_{c,e}-u_{c,e}^q)}(k+2\ell n) + \sum_{\substack{k \ge n \\ k < -n}} \beta(k)\widehat{(u_{c,e}-u_{c,e}^q)}(k)e^{\pi i k x}.
\end{align*}
Thus, its magnitude can be bounded above by
\begin{align}
\sum_{k=-n}^{n-1}|\beta(k)|\sum_{\ell=1}^{\infty}  |\widehat{(u_{c,e}-u_{c,e}^q)}(k+2\ell n) + \widehat{(u_{c,e}-u_{c,e}^q)}(-k-2\ell n)| + \sum_{|k| \ge n} |\beta(k)||\widehat{(u_{c,e}-u_{c,e}^q)}(k)|.
 \label{eq:err2}
\end{align}
The convergence rate, therefore, depends not only on the rates at which Fourier coefficients of $u_{e,c}$ and $u_{c,e}-u_{c,e}^q$ decay but also on the decay rate of $\beta(k)$. 

\begin{figure}[b]
	\centering
	\begin{subfigure}[$g(x) = |x|^{-4/5}$]
		{\includegraphics[width=0.3\textwidth,trim={1.5cm 0.8cm 1.4cm 0.8cm},clip]{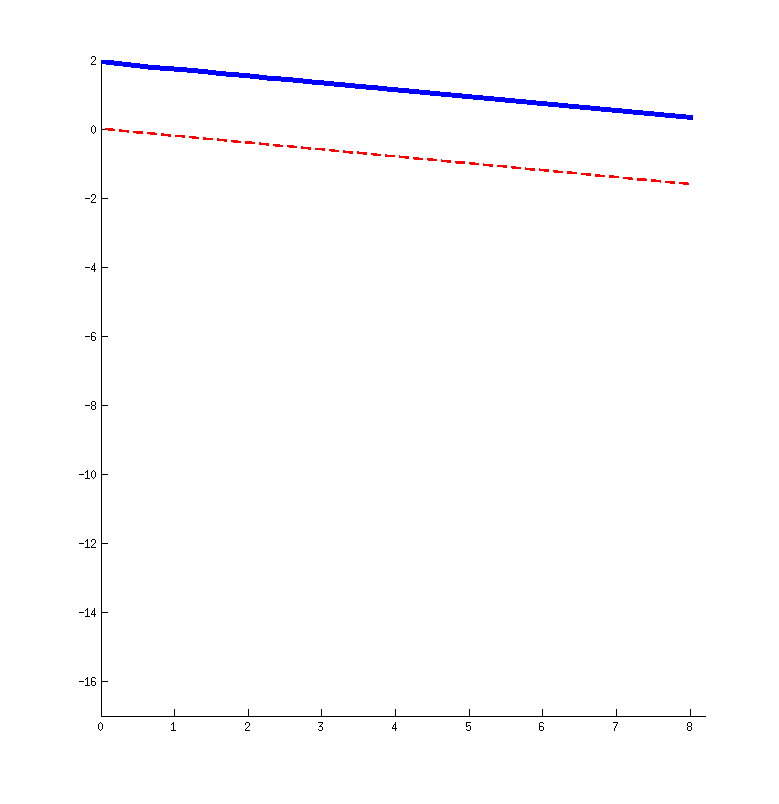}}
	\end{subfigure}
	~
	\begin{subfigure}[$g(x) = |x|^{-1/2}$]
		{\includegraphics[width=0.3\textwidth,trim={1.5cm 0.8cm 1.4cm 0.8cm},clip]{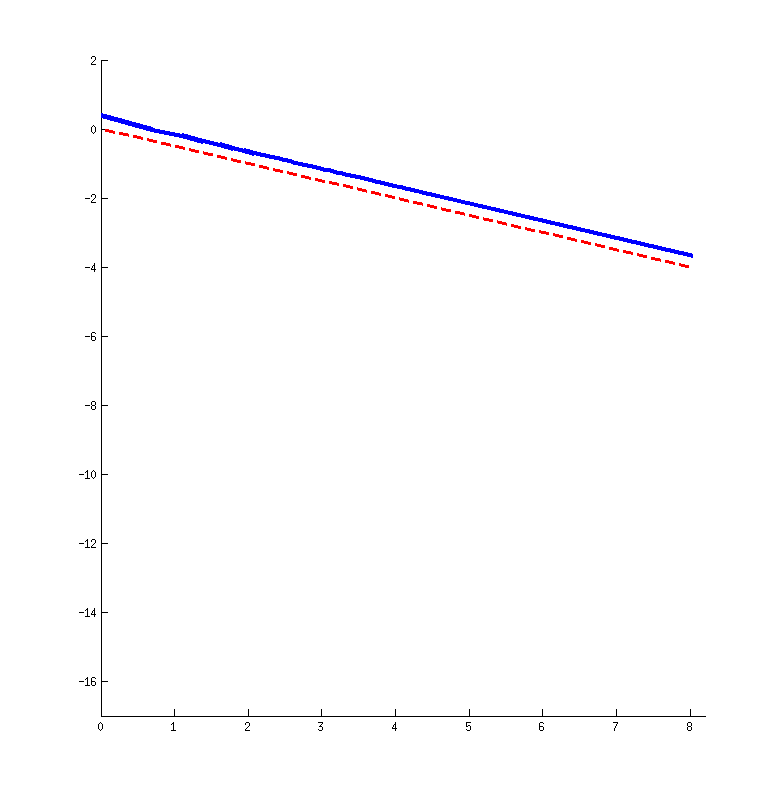}}
	\end{subfigure}
	~
	\begin{subfigure}[$g(x) = |x|^{1/2}$]
		{\includegraphics[width=0.3\textwidth,trim={1.5cm 0.8cm 1.4cm 0.8cm},clip]{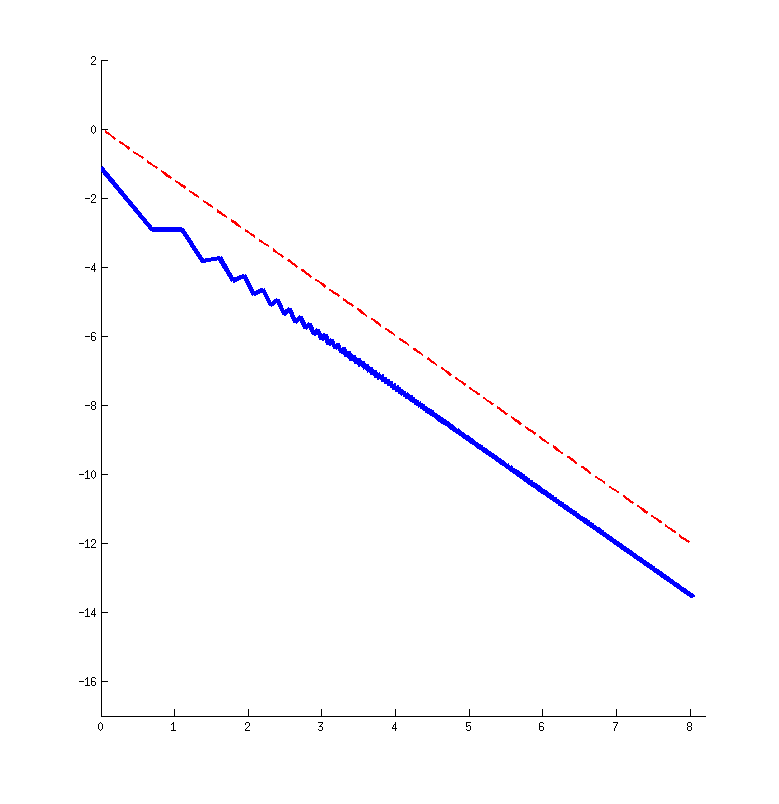}}
	\end{subfigure}
	~
	\begin{subfigure}[$g(x) = |x|^{4/5}$]
		{\includegraphics[width=0.3\textwidth,trim={1.5cm 0.8cm 1.4cm 0.8cm},clip]{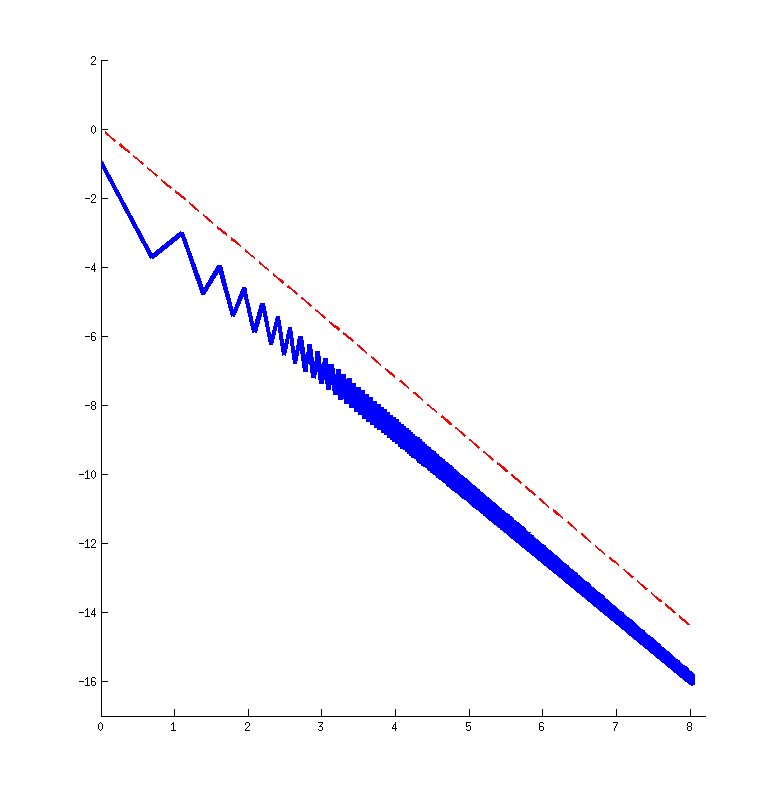}}
	\end{subfigure}
	~
	\begin{subfigure}[$g(x) = |x|^{3}$]
		{\includegraphics[width=0.3\textwidth,trim={1.5cm 0.8cm 1.4cm 0.8cm},clip]{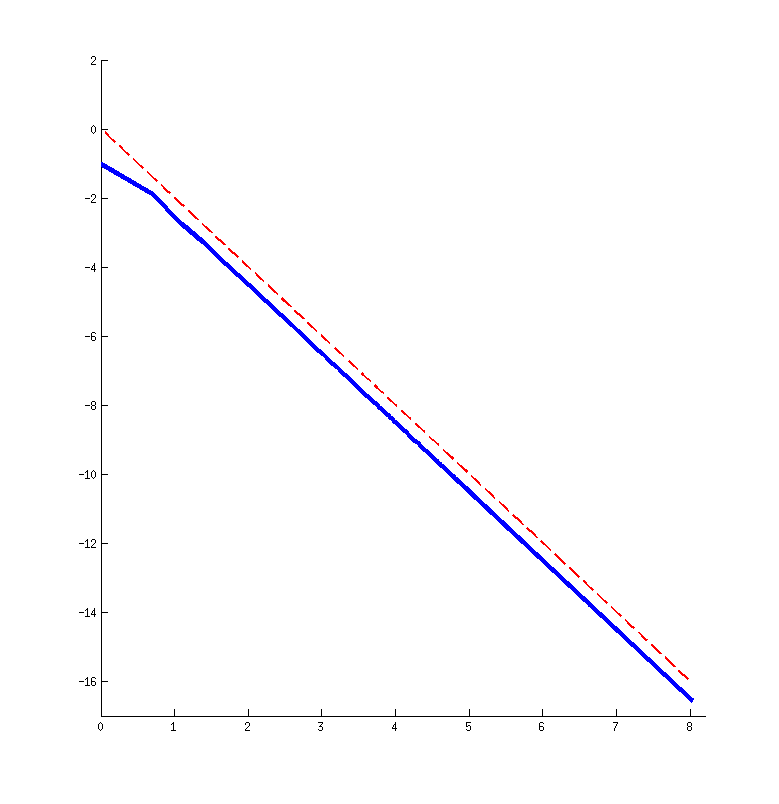}}
	\end{subfigure}
	~
	\begin{subfigure}[$g(x) = \log |x|$]
		{\includegraphics[width=0.3\textwidth,trim={1.5cm 0.8cm 1.4cm 0.8cm},clip]{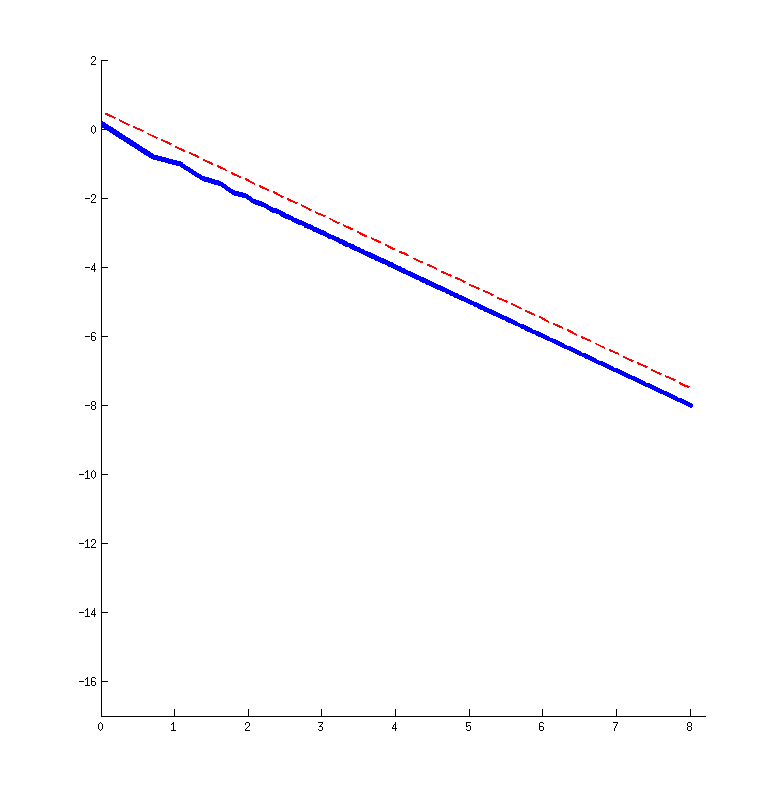}}
	\end{subfigure}
	\caption{ Plots showing $\log|\beta(k)|$ against $\log(k)$ for kernels of the form $g(x) = |x|^{\gamma}$ for different $\gamma$ values as well as $g(x) = \log |x|$; for each subfigure, a dashed line with slope as predicted by \Cref{lm:-BK} is shown for reference. Recall that for the bottom right figure, expected slope is $-1$ whereas it is $-\min\{1+\gamma,2\}$ for figures with $g(x) = |x|^{\gamma}$.  
	}
	\label{fig:-betakwol}
\end{figure}

\begin{lemma} \label{lm:-BK}
	For a $\gamma \in (-1,\infty)$ and $g(x) = |x|^\gamma$, there exists a positive constant $B_\gamma$ such that 
	\[
	\left|\beta(k)\right| \le B_\gamma |k|^{-\min\{1+\gamma,2\}}
	\]
	for all $k \ne 0$. 
	Similarly, for $g(x) = \log|x|$, we have
	\[
	|\beta(k) | \le  2 |k|^{-1}
	\]
	for all $k \ne 0$.
\end{lemma}
\begin{proof} We begin by proving the result for $g(x) = |x|^{-\gamma}, \gamma \in (-1,\infty).$ Note that we only need to show that the result holds for $k > 0$ as $\beta(-k) = \overline{\beta(k)}$.
We first assume that $\gamma \in (-1,0)$, $k > 0$, and observe that
\begin{align}\label{eq:nbk}
\beta(k) &= 
\frac{2}{k^{1+\gamma}} \vartheta(k),
\end{align}
where
\begin{align*}
\vartheta(k) = \int_{0}^{k} t^{\gamma} \cos{\pi t}\,dt
&= \sum_{\ell=0}^{ \lfloor k/2 \rfloor - 1}  \int_{0}^{1} \left[(t+2\ell)^{\gamma} - (t+2\ell+1)^{\gamma}\right] \cos{\pi t}\,dt + \int_{2\lfloor k/2\rfloor}^{k} t^{\gamma}\cos{\pi t}\,dt.
\end{align*}
The result now follows from \cref{eq:nbk} and the estimate
\begin{align*}
|\vartheta(k)| &\le \frac{3-2^{1+\gamma}}{1+\gamma}  + \sum_{\ell=1}^{ \lfloor k/2 \rfloor - 1} \left\{ (2\ell)^{\gamma} - (2+2\ell)^{\gamma} \right\} 
\le \frac{3-2^{1+\gamma}}{1+\gamma}  + 2^{\gamma}.
\end{align*}

	For $\gamma = 0$, the result follows trivially from the definition, as $\beta(k) = \int_{-1}^1 e^{\pi i k \rho}\,d\rho = 0$
for all $k \ne 0$. 

Now, for $\gamma \in (0,1], k>0$, we have
\begin{align}
\beta(k) 
= -\frac{2\gamma}{\pi k^{1+\gamma}}\int_0^k t^{\gamma-1}\sin \pi t\,dt.
\label{eq:nnbk}
\end{align}
An argument analogous to the one used above shows that the integral in \cref{eq:nnbk} is bounded by a constant that depends only on $\gamma$, thus, establishing the result for $\gamma \in (0,1]$.

Finally, for $\gamma > 1$, the result follows from
\[
\beta(k) 
= (-1)^k\frac{2\gamma}{\pi^2 k^2} +\frac{2\gamma(\gamma-1)}{\pi^2 k^2}\int_0^1 \rho^{\gamma-2}\cos \pi k \rho\,d\rho.
\]
Now, for $g(x) = \log |x|, k>0$, we observe that
 \begin{align}\label{eq:nbkl}
\beta(k) 
&= \frac{2}{k} \zeta(k),
\end{align}
where
\begin{align*}
\zeta(k) &= \int_{0}^{k} \log t \cos{\pi t}\,dt.
\end{align*}
Now, from the following estimate, result follows,
\begin{align*}
|\zeta(k)|
&\le 1 + \sum_{\ell=1}^{ k-\lfloor k/2 \rfloor - 1}  \int_{0}^{1} \log \left(\frac{t+2\ell-(k-2\lfloor k/2\rfloor)+1}{t+2\ell-(k-2\lfloor k/2\rfloor)}\right) \,dt \\
&\le 1 + \sum_{\ell=1}^{ \infty}  \log \left(\frac{2+2\ell-(k-2\lfloor k/2\rfloor)}{2\ell-(k-2\lfloor k/2\rfloor)}\right) 
 = 1 - \log \left(2-(k-2\lfloor k/2\rfloor)\right) \le 1.
\end{align*}
\end{proof}
We present some examples in \Cref{fig:-betakwol} to demonstrate that the behavior of $\beta(k)$ established in \Cref{lm:-BK} is indeed observed in practice. 

Next results look at the behavior of how Fourier coefficients decay for the difference between two different extensions of a given function.
\begin{lemma} \label{lm:-disc_err}
Let $f_1$ and $f_2$ be two extensions of $f : [0,1] \to \mathbb{R}$ that correspond to matrices $F_1 = (f^1_{jm})_{0 \le j \le 1, 0 \le m \le r}$ and $F_2 = (f^2_{jm})_{0 \le j \le 1, 0 \le m \le r}$ respectively, that is,
\[
f_1(x) = \begin{cases} f(x), & x \in [0,1], \\ p(F_1)(x), & x \in [-1,0),\end{cases}
\ \text{ and }\ \ f_2(x) = \begin{cases} f(x), & x \in [0,1], \\ p(F_2)(x), & x \in [-1,0).\end{cases}
\]
If $F_1$ and $F_2$ are such that their entries satisfy $f^1_{jm} = f^2_{jm} = f^{(m)}(j), j = 0,1$ and $m = 0, \ldots, s$ for some $0 \le s \le r$, then there exists a positive constant $C$ such that
\begin{align} 
\label{eq:disc_err1}
|\widehat{(f_1 - f_2)}(0)| &\le C  \Vert F_1 - F_2 \Vert_{max} \\
\label{eq:disc_err2}
|\widehat{(f_1 - f_2)}(k)| &\le C  |k|^{-(2+s)} \Vert F_1 - F_2 \Vert_{max}
\end{align}
for all $k \ne 0$, where for a matrix $F =(f_{jm})$, $\Vert F \Vert_{max} = \max_{jm} |f_{jm}|$.
\end{lemma}
\begin{proof} If $s = r$, then \eqref{eq:disc_err1} and \eqref{eq:disc_err2} holds with equality as both sides are zero. For $s < r$, the result follows from
\begin{align*}
\widehat{(f_1-f_2)}(0) = \frac{1}{2}\int_{-1}^{0} \left( \sum_{m=0}^r (f_{0m}^1-f_{0m}^2) p_m^0(x) + \sum_{m=0}^r (f_{1m}^1-f^2_{1m}) p_m^1(x)\right)\,dx
\end{align*}
and, for $k \ne 0$, from
\begin{align*}
&\widehat{(f_1-f_2)}(k) = \frac{1}{2}\int_{-1}^0 p(F_1-F_2)(x)e^{-\pi i k x}\,dx = \\ &\frac{1}{2}\frac{1}{(-\pi i k)^{s+2}}\left( \sum_{m=0}^r (f_{0m}^1-f_{0m}^2) (p_m^0)^{(s+1)}(0) + \sum_{m=0}^r (f_{1m}^1-f_{1m}^2) (p_m^1)^{(s+1)}(0)\right) - \\ 
& \frac{1}{2}\frac{(-1)^k}{(-\pi i k)^{s+2}}\left( \sum_{m=0}^r (f_{0m}^1-f_{0m}^2) (p_m^0)^{(s+1)}(-1) + \sum_{m=0}^r (f^1_{1m}-f^2_{1m}) (p_m^1)^{(s+1)}(-1)\right) - \\
&\frac{1}{2}\frac{(-1)^k}{(-\pi i k)^{s+2}}\int_{-1}^{0} \left( \sum_{m=0}^r (f_{0m}^1-f_{0m}^2) (p_m^0)^{(s+2)}(x) + \sum_{m=0}^r (f_{1m}^1-f^2_{1m}) (p_m^1)^{(s+2)}(x)\right) e^{-\pi i k x}\,dx.
\end{align*}	
\end{proof}

The previous lemma thus implies that $\widehat{(u_{c,e}-u_{c,e}^q)}(k)$ decays as the rate $|k|^{-2}$. 
Finally, combining the estimates \cref{eq:err1} and \cref{eq:err2} with results in \Cref{lm:-BK} and \Cref{lm:-disc_err}, we arrive at the following upper bound for the numerical integration error using \cref{eq:na}.
\begin{theorem}\label{th: convergence} Let $g : \mathbb{R} \to \mathbb{R}$ be an even and absolutely integrable function such that $\beta(k) \le B |k|^{-(1+\gamma)}, k \ne 0,$ for some $B > 0$, $\gamma > -1$ and there exists a positive integer $n_q$ such that $\Vert U - U^q \Vert_{max} \le B n^{-q}$, for some $q > 0$ and all $n \ge n_0$. If $u \in C([0,1])$ be such that $|\widehat{u_{c,e}}(k) + \widehat{u_{c,e}}(-k)| \le B|k|^{-\delta}, k \ne 0$, for some $\delta > 1$ then, 
there exists a positive constant $C$ such that, for
the sequence of approximations $(A_nu), n \ge n_0$ as defined in \cref{eq:quad} to $Au$ given in \cref{eq:-Ie}, we have the error estimate
\begin{align*}
\Vert Au - A_nu \Vert_{\infty} &\le Cn^{-\min\{2+q+\gamma,2+q,\delta,\delta+\gamma\}}.
\end{align*}
In particular, if $u \in C^{0,\alpha}$ with $\alpha \in (\max\{0,-\gamma\},1]$, then
\[
\Vert Au - A_nu \Vert_{\infty} \to 0 \ \ \text{ as }\ \ n \to 0.
\]
\end{theorem}

We recall that if $u\in C^\infty([0,1])$, then $u_{c,e}\in C^{r,1}([0,1])$ and is piecewise smooth. In this case, for $k \ne 0$, we have
\[
|\widehat{u_{c,e}}(k) + \widehat{u_{c,e}}(-k)| \le \begin{cases}
B|k|^{-(2+r)},  & \text{if $r$ is even}, \\
 B|k|^{-(3+r)},  & \text{if $r$ is odd}.
\end{cases}
\]
Therefore, the error estimate in \Cref{th: convergence} reduces to 
\begin{align*}
\Vert Au - A_nu \Vert_{\infty} &\le \begin{cases}
Cn^{-\min\{2+q+\gamma,2+q,2+r,2+r+\gamma\}},  & \text{if $r$ is even}, \\
 Cn^{-\min\{2+q+\gamma,2+q,3+r,3+r+\gamma\}},  & \text{if $r$ is odd}.
\end{cases}
\end{align*} 
In particular, we note that when $\gamma < 0$, we have
\[
\Vert Au - A_nu \Vert_{\infty} \le \begin{cases}
Cn^{-(2+\gamma+\min\{q,r\})},  & \text{if $r$ is even}, \\
Cn^{-(2+\gamma+\min\{q,1+r\})},  & \text{if $r$ is odd}.
\end{cases}
\]
Whereas, for $\gamma \ge 0$, we have
\[
\Vert Au - A_nu \Vert_{\infty} \le \begin{cases}
Cn^{-(2+\min\{q,r\})},  & \text{if $r$ is even}, \\
Cn^{-(2+\min\{q,1+r\})},  & \text{if $r$ is odd}.
\end{cases}
\]

 \section{Numerical computation of $\beta(k)$}\label{sec:-Bk}

The primary difficulty in the numerical evaluation of $\beta(k)$ is due to the kernel singularity at $r=0$. To overcome this challenge,
we use a change of variable of the form $r=\tau^{M}$ with an odd $M$.
The transformed integral then reads
\begin{align*}
\beta(k) = M \int_{-1}^{1}\tau^{M-1} g(\tau^M)) e^{\pi ik\varrho(\tau^M)} \,d\tau .
\end{align*}
We note that the Jacobian, $\tau^{M-1} $ of the transformation renders the integrand $(M-1)$ times continuously differentiable in $\tau$
if $g(x) = \log(|x|)$, whereas for $g(r) = |r|^{\gamma}$, the integrand is $M'$ times continuously differentiable if $\lfloor(1+\gamma)(M+1)\rfloor-1 \in [M',M'+1)$ for an $M' \in \mathbb{N}$. Thus, the integral defining $\beta(k)$
can now be approximated to a high-order accuracy using an appropriate high-order quadrature. For example, in the calculations that we present in this text, we have employed the
Clenshaw Curtis quadrature. For completeness, we recall that the corresponding quadrature points $\{x^{cc}_j: j = 0,\ldots,n_{cc}\}$ and weights $\{\omega^{cc}_j: j = 0,\ldots,n_{cc}\}$
over $[-1,1]$ is given by 
\[
x^{cc}_j = \cos\left(\frac{(2j+1)\pi}{2n_{cc}}\right),\ \ \  j = 0,\ldots, n-1  
\]
and 
\begin{align*}
\omega^{cc}_j = \frac{2}{n_{cc}} \sideset{}{'}\sum_{k=0}^{n_{cc}} \vartheta_k \cos\left(\frac{(2j+1)k \pi}{2n_{cc}}\right),
\end{align*}
where 
\begin{align*}
\vartheta_k = 
\begin{cases}
\frac{-2}{k^2-1}, & \text{if k is even}, \\
0 , & \text{if k is odd}. 
\end{cases}
\end{align*}

%

We demonstrate the effectiveness of this numerical approach through a sequence of computational examples where we compute $\beta(k)$ covering a range of $k$ values. In particular, in Tables \ref{table:-betacompR_16}, \ref{table:-betacompR_256} 
and \ref{table:-betacomp_1024}, we present a convergence study for the numerical computation of $\beta(k)$, where $k$ takes values $16$, $256$ and  $1024$, respectively. We observe from these tables that our numerical scheme indeed computes $\beta(k)$
with high-order accuracy. As expected, to maintain a desired level in the 
accuracy across all frequencies, we need to suitably increase $n_{cc}$ with increasing values of $k$.

\begin{table}[t]
	\begin{center}
		\begin{tabular}{c|c|c} \hline
			$n_{cc}$ & $\beta(16)$  & Error   \\
			\hline
			$2^{5}$  & $- 0.41424931329196118-0.3475645408951408i$ & $4.5342\times10^{-1}$ \\ 
			\hline
			$2^{6}$  & $- 0.06045174828858752-0.0146128696936386i$ & $5.2144\times10^{-2}$  \\
			\hline
			$2^{7}$  & $- 0.11231288276991235-0.0091768451503311i$ & $3.531\times10^{-7}$ \\ 
			\hline
			$2^{8}$  & $- 0.11231323588089342-0.0091771070920223i$ & $8.9342\times10^{-13}$ \\ 
			\hline
		\end{tabular}
		\caption{Numerically computed values of $\beta(16)$ for various Chenshaw Curtis quadrature grids of size $n_{cc}$ when $g(x) = \log(|x|)$. The corresponding ``exact value'', for comparison, is  $-0.11231323588086038-0.00917710709201628i$ which is obtained to 16 digits of accuracy using Mathematica.} \label{table:-betacompR_16}
	\end{center}
\end{table}

\begin{table}[b]
	\begin{center}
		\begin{tabular}{c|c|c} \hline
			$n_{cc}$ &  $\beta(256)$  & Error   \\
			\hline
			$2^{9}$  & $- 0.1032298166113593 - 0.043479426501543391i$& $1.0534\times10^{-1}$ \\ 
			\hline
			$2^{10}$  & $- 0.00698157703241292487-0.000641865536413126i$ & $1.2973\times10^{-8}$ \\ 
			\hline
			$2^{11}$ & $- 0.0069815658535057186-0.000641872133701744i$ & $8.9492\times10^{-13}$ \\ 
			\hline
		\end{tabular}
		\caption{Numerically computed values of $\beta(256)$ for various Chenshaw Curtis quadrature grids of size $n_{cc}$ when $g(x) = \log(|x|)$. The corresponding ``exact value'', for comparison, is  $-0.006981565853453448-0.000641872133719719i$ as obtained to 16 digits of accuracy using Mathematica.} \label{table:-betacompR_256}
	\end{center}
\end{table}
\begin{table}[t]
	\begin{center}
		\begin{tabular}{c|c|c} \hline
			$n_{cc}$ & $\beta(1024)$  & Error   \\
			\hline
			$2^{10}$  & $- 0.024051290624620809-0.0849763910282477i$ & $8.7692\times10^{-2}$ \\ 
			\hline
			$2^{11}$  & $- 0.010792170763582126-0.0211729731252266i$ & $2.2865\times10^{-2}$  \\
			\hline
			$2^{12}$  & $- 0.001744903243373586-0.0001613157849477i$ & $9.7419\times10^{-11}$ \\ 
			\hline
			$2^{13}$  & $- 0.001744903278200884-0.0001613158759463i$ & $1.2662\times10^{-13}$\\ 
			\hline
		\end{tabular}
		\caption{Numerically computed values of $\beta(1024)$ for various Chenshaw Curtis quadrature grids of size $n_{cc}$ when $g(x) = \log(|x|)$. The corresponding ``exact value'', for comparison, is  $-0.00174490327807993-0.00016131587598406i$ which is obtained to 16 digits of accuracy using Mathematica.} \label{table:-betacomp_1024}
	\end{center}
\end{table}

\begin{remark}  
	The values for $\beta(k)$ are computed apriori with high accuracy and stored for their use in schemes given in \cref{eq:conv}, \cref{eq:quad} and thus do not add significant online computational burden at the numerical integration stage. Nevertheless, upto moderately large values of $k$, they can be obtained using an FFT based fast scheme in a reduced computational time. For example, if $g(r) = |r|^{\gamma}$, we have
	\begin{align*}
	\beta(k) &=  \int_{-1}^{1}|\rho|^{\gamma} e^{\pi ik\rho} d\rho
	=  2\int_{0}^{1}\rho^{\gamma} \cos(\pi k \rho) d\rho,
	\end{align*}
	which, after, $2M$ repeated integrating by parts, reads
	\begin{align} \label{eq:fast_beta}
	\beta(k) = 2 (-1)^k \sum_{\ell=1}^M (-1)^{\ell-1} 
	\frac{(\pi k)^{2(\ell-1)}}{\prod\limits_{j=1}^{2\ell-1}(j+\gamma)}  + 
	\frac{(-1)^M (\pi k)^{2M}}{\prod\limits_{j=1}^{2M}(j+\gamma)} \int_{-1}^{1}|\rho|^{2M+\gamma} e^{\pi ik\rho} d\rho.
	\end{align}
	An application of $n_{cc}$-point trapezoidal rule with endpoint corrections for the numerical integration in \cref{eq:fast_beta} yields
	
	\begin{align} \label{eq:fast_beta_1}
	\beta(k) &= 2 (-1)^k \sum_{\ell=1}^M (-1)^{\ell-1} 
	\frac{(\pi k)^{2(\ell-1)}}{\prod\limits_{j=1}^{2\ell-1}(j+\gamma)}  + \sum_{m=1}^{M-1} \frac{B_{2m}}{(2m!)} \frac{\psi^{(2m-1)}(-1)-\psi^{(2m-1)}(1)}{n_{cc}^{2m}} \\ &+ 
	\frac{(-1)^M (\pi k)^{2M}}{\prod\limits_{j=1}^{2M}(j+\gamma)} 
	\frac{1}{n_{cc}}  \sum_{m=0}^{n_{cc}-1} \psi(\rho_m), \nonumber
	\end{align}
	where $\psi(x) = |x|^{2p+\gamma} e^{\pi ikx}$, $x \in [-1,1]$, $\rho_m = -1 + 2m/n_{cc}$, $m= 1 \ldots n_{cc}-1$ and $B_m$ are Bernoulli numbers. Note that the summation in the last term of \cref{eq:fast_beta_1} can be evaluated for all $|k| < n_{cc}/2$ in $O(n_{cc} \log n_{cc})$ operations using the FFT while evaluation of first and second correction terms in \cref{eq:fast_beta_1} require $O(n_{cc})$ operations in total. 
	
\end{remark}

\section{Compactly supported integrands}\label{sec:-csi}
Many important instances of integral operators come with density functions whose support is contained within the domain of integration. A prominent source of such examples is the integral equation formulations of penetrable scattering problem where the material inhomogeneity is localized in a region bounded by a homogeneous surrounding. The proposed integration scheme, when applied to integrals with  smooth and compactly supported integrands, simplifies even further. For instance, the extension of $u$ from $[0,1]$ to $[-1,1]$ is trivial with the zero polynomial doing the job. Consequently, no correction is required as both $C_L(U)(x)$ and $C_R(U)(x)$ identically take the zero value. The numerical scheme in this case, therefore, reads
\[
(A_nu)(x) = \sum_{j=0}^{n-1} w_j^n(x) u_{j},
\]
%
  where
  \[
  w_j^n(x) = \frac{1}{2n} \sum_{k=-n}^{n-1} \beta(k)e^{\pi ik\left(x-j/n\right)}\ \ \ \text{ with }\ \ \ 
  \beta(k) =  \int_{-1}^{1} g(r) e^{\pi ik\rho} d\rho.
  \]
  The form that corresponds to \cref{eq:conv} simply reads
  \[
  (A_nu)(x) = \sum_{k=-n}^{n-1} \beta(k)\hat{u}_{c,n}^q(k)
  \]
  and can be implemented in a straightforward manner using FFT and its inverse using the pre-computed $\beta(k)$ values. For $u \in C^{\infty}([0,1])$, the rate of convergence depends directly on order with which $u$ vanishes at the boundary. A couple of numerical experiments (see \cref{ex:compact_supp}) to exemplify this has been included in the next section.
  

\section{Numerical Results} \label{sec:-ne}
\label{sec:Numerical Results}
In this section, we demonstrate high-order 
convergence and accuracy of our quadrature through several computational 
examples. 
We compute the numerical order of convergence of our algorithm which relates to the rate at which the error in the approximation of integral decreases as the discretization scale decreases according to
the following formula:
\begin{equation} 
 \text{Order} = \log_{2}\left(\frac{\max_{0\le x_i\le1}|A(x_i)-A_n(x_i)|}{\max_{0\le x_i\le1}|A(x_i)-A_{2n}(x_i)|}\right).
\end{equation}
The relative error reported in this section are computed as
\begin{equation} 
\varepsilon_{\infty} =   \frac{\max_{0\le x_i\le1}|A_{exact}(x_i) - A_n(x_i)|}{\max_{0\le x_i\le1}|A_{exact}(x_i)|} .
\end{equation}
To corroborate the decay rate, we plot $\log_{10}(\varepsilon_{\infty})$ against $\log_{10}(h),\, h=1/n$, where the corresponding exact value of the integral is obtained to 16 digits of accuracy using Mathematica whenever it can not be obtained analytically. 

\begin{example}
In the first example, we consider the integral operator $A$, defined in \eqref{eq:-Ie}, with singular function $g(x) = |x|^\gamma$ and $u(x)=x$. In this case, $(Au)(x)$ can be obtained analytically and  is given by
\begin{align}\label{eq:exact}
(Au)(x)= \frac{x^{2+\gamma}}{2+3\gamma+\gamma^2} + \frac{(1-\gamma)^{1+\gamma}(1+\gamma+x)}{(1+\gamma)(2+\gamma)}.
\end{align}

We compare the numerically computed integrals with \eqref{eq:exact} when $\gamma=-4/5$. The results have been shown in \Cref{fig:-error2}.
 Note that as $u(x)$ is a linear function, $U^q = U$ for $q\ge1$. Therefore, there is no effect of the parameter $q$ on the convergence as $q$ increases. To  corroborate this, we fix the parameter $q$ at 1 and 3, and vary the smoothness parameter $r$ from 1 to 4 and show the corresponding results in \Cref{fig:-error2}a and \Cref{fig:-error2}b respectively. These plots clearly demonstrate the effect of $r$ on the convergence rate.
\end{example}

\begin{figure}[t]
	\centering
	\begin{subfigure}[$q=1$]
		{\includegraphics[width=0.45\textwidth,trim={1.5cm 0.7cm 1.4cm 0.8cm},clip]{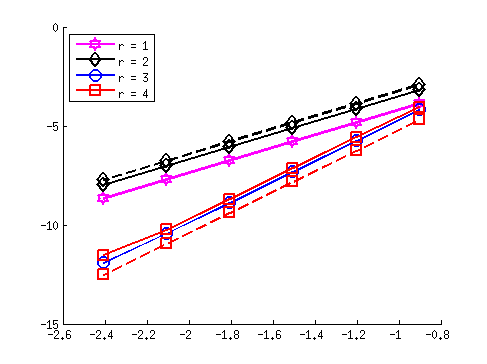}}
	\end{subfigure}
	~
	\begin{subfigure}[$q=3$]
		{\includegraphics[width=0.45\textwidth,trim={1.5cm 0.7cm 1.4cm 0.8cm},clip]{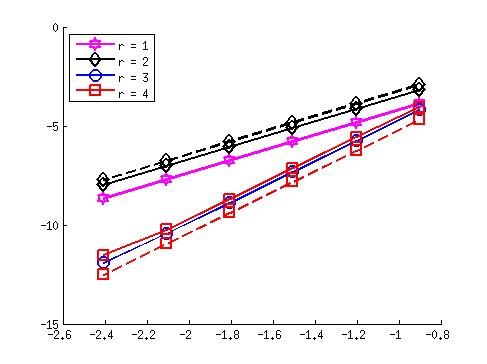}}
	\end{subfigure}

	\caption{Plots showing $\log_{10}(\varepsilon_{\infty})$ against $\log_{10}(h),\, h=1/n$ for the kernel $g(x) = |x|^{-4/5}$ and $u(x)=x.$ In each subfigure, dashed line with slopes as predicted by \Cref{th: convergence} are shown for reference.  
	}
	\label{fig:-error2}
\end{figure}

\begin{example}
In the next example, we present the numerical results to demonstrate the  effect of smoothness parameter $r$ and order of the boundary derivative approximation $q$ on the convergence of our numerical scheme. 
 We consider the integral operator $A$, defined in \eqref{eq:-Ie}, with singular function $g(x) = |x|^\gamma,\, \gamma = -4/5$ and $u(x)=\cos(x)$. In \Cref{fig:-error1}a and \Cref{fig:-error1}b, we fix the smoothness parameter at $r=2$ and $r=3$ respectively and vary the parameter $q$ from 1 to 4 in both cases. It can be seen in \Cref{fig:-error1}a and \Cref{fig:-error1}b that the integration scheme performs well and the expected order  of convergence is achieved in each case. Subsequently, in \Cref{fig:-error1}c and \Cref{fig:-error1}d, we fix the parameter $q$ and vary the smoothness parameter $r$ to observe its effect on the convergence rate. The plots clearly show that the numerical errors converge as expected.
\end{example}

\begin{figure}[t]
	\centering
	\begin{subfigure}[$r=2$]
		{\includegraphics[width=0.45\textwidth,trim={1.2cm 0.7cm 1.4cm 0.8cm},clip]{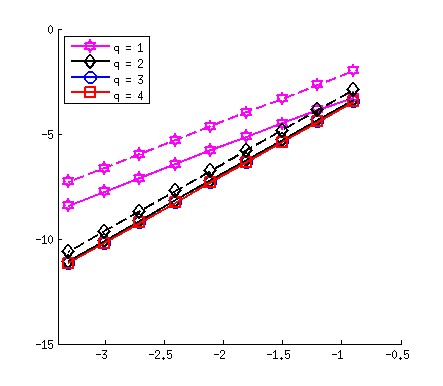}}
	\end{subfigure}\label{fig:-error1a}
	~
	\begin{subfigure}[$r=3$]
		{\includegraphics[width=0.45\textwidth,trim={1.2cm 0.7cm 1.4cm 0.8cm},clip]{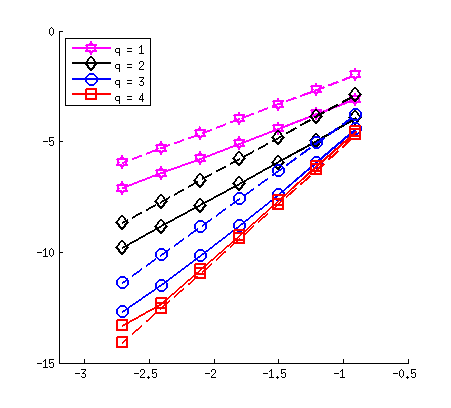}}
	\end{subfigure}

	~
	\begin{subfigure}[$q=2$]
		{\includegraphics[width=0.45\textwidth,trim={1.2cm 0.7cm 1.4cm 0.8cm},clip]{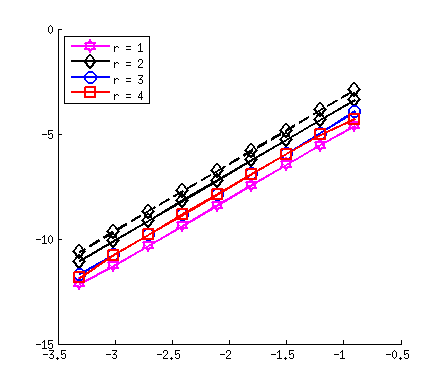}}
	\end{subfigure}
	~
	\begin{subfigure}[$q=3$]
		{\includegraphics[width=0.45\textwidth,trim={1.2cm 0.7cm 1.4cm 0.8cm},clip]{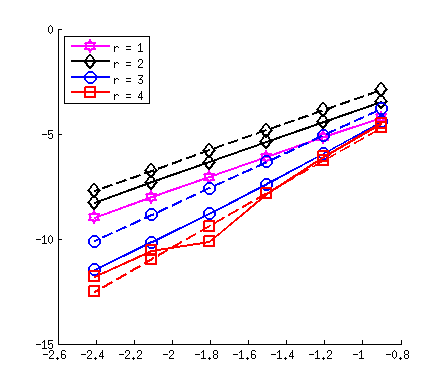}}
	\end{subfigure}

	\caption{Plots showing $\log_{10}(\varepsilon_{\infty})$ against $\log_{10}(h),\, h=1/n$ for  $g(x) = |x|^{-4/5}$ and   $u(x)=\cos(x);$ dashed lines with slopes as predicted by \Cref{th: convergence} are shown for reference.   
	}
	\label{fig:-error1}
\end{figure}

\begin{example}
Next, we test our numerical integration scheme for some other kernel, for example $g(x) = \log(|x|)$ and $g(x) = |x|^\gamma,\, \gamma=1/2$. 
 To show the effect of parameter $r$ over the convergence of the integration scheme, we consider $u(x)=x$ and fix $q$ at 1 and vary the smoothness parameter $r$. It can be seen in \Cref{fig:-error3}a and \Cref{fig:-error4}a that the integration scheme performs well and the expected order of convergence is achieved in each case. Next, in \Cref{fig:-error3}b and \Cref{fig:-error4}b, we consider $u(x)=\cos(x)$ and fix $r$ at 3 and vary $q$ to observe its effect on the convergence rate. These plots clearly demonstrate the effect of $q$ on the convergence rate. 
\end{example}

\begin{figure}[t]
	\centering
	\begin{subfigure}[$q=1,\,u(x)=x$]
		{\includegraphics[width=0.45\textwidth,trim={1.5cm 0.8cm 1.4cm 0.8cm},clip]{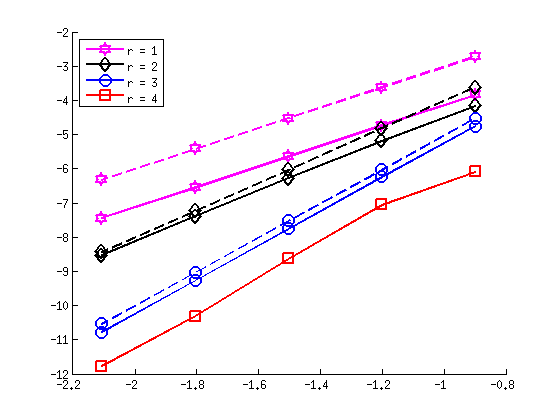}}
	\end{subfigure}
	~
	\begin{subfigure}[$r=3,\, u(x)=\cos(x)$]
		{\includegraphics[width=0.45\textwidth,trim={1.5cm 0.8cm 1.4cm 0.8cm},clip]{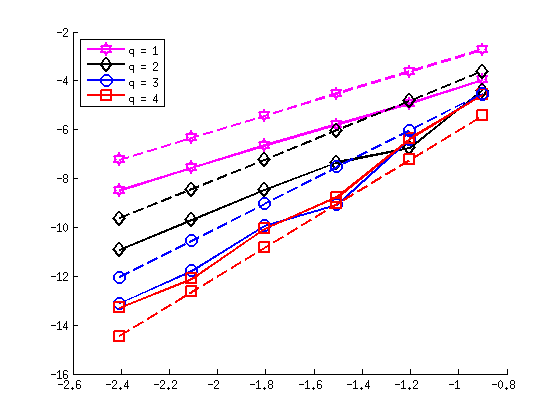}}
	\end{subfigure}

	\caption{Plots showing $\log_{10}(\varepsilon_{\infty})$ against $\log_{10}(h),\, h=1/n$ for the kernel $g(x) = \log(|x|)$; dashed lines with expected slopes are shown for reference. 
	}
	\label{fig:-error3}
\end{figure}

\begin{figure}[t]
	\centering
	\begin{subfigure}[$q=1,\, u(x)=x$]
		{\includegraphics[width=0.45\textwidth,trim={1.5cm 0.8cm 1.4cm 0.8cm},clip]{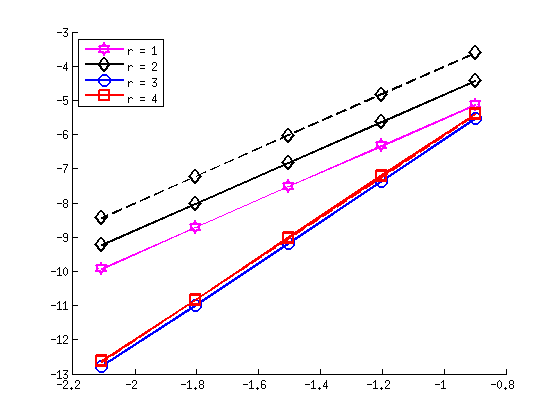}}
	\end{subfigure}
	~
	\begin{subfigure}[$r=3,\, u(x)=\cos(x)$]
		{\includegraphics[width=0.45\textwidth,trim={1.5cm 0.8cm 1.4cm 0.8cm},clip]{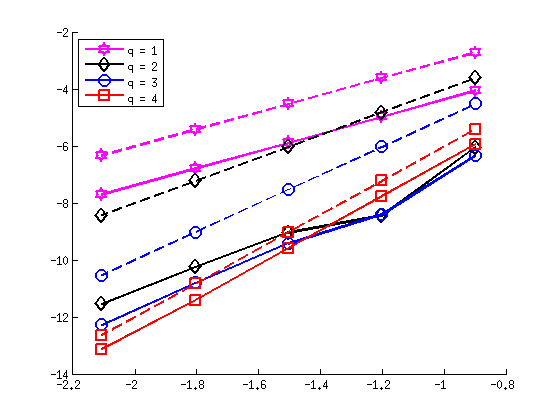}}
	\end{subfigure}

	\caption{
	Plots showing $\log_{10}(\varepsilon_{\infty})$ against $\log_{10}(h),\, h=1/n$ for the kernels $g(x) = |x|^\gamma,\, \gamma=1/2$; dashed lines with expected slopes are shown for reference. 
	}
	\label{fig:-error4}
\end{figure}

\begin{example}
	\label{ex:compact_supp}
In the last example, we demonstrate the performance of our numerical integration scheme, for functions whose support is contained in $[0,1]$. Toward this, we first consider $u =e^{-(\frac{x-1/2}{\sigma})^2}$ with $\sigma=0.01$. As $u$ vanishes to high order at boundary points $x = 0$ and $x = 1$, the numerical scheme exhibits a super algebraic convergence as seen in \Cref{fig:-error5}. We note that this is in contrast with the methodology used in \cite{Marin} where expected rate of convergence is algebraic for similarly compactly supported integrands. Finally, to demonstrate the effect of smoothness parameter $r$ on the rate of convergence,  we choose a smooth compactly supported function $u = x^3(1-x)^3$ that satisfies $u^{(\ell)}(0)= 0 =u^{(\ell)}(1)$ for $\ell<3$. The corresponding convergence study is shown in \Cref{fig:-error6}.  
\end{example}

\begin{figure}[t]
	\centering
	\begin{subfigure}[$q=1 $]
		{\includegraphics[width=0.45\textwidth,trim={1.5cm 0.8cm 1.4cm 0.8cm},clip]{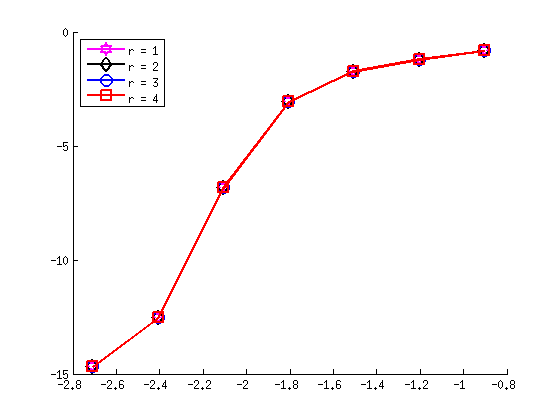}}
	\end{subfigure}
	~
	\begin{subfigure}[$r=3$]
		{\includegraphics[width=0.45\textwidth,trim={1.5cm 0.8cm 1.4cm 0.8cm},clip]{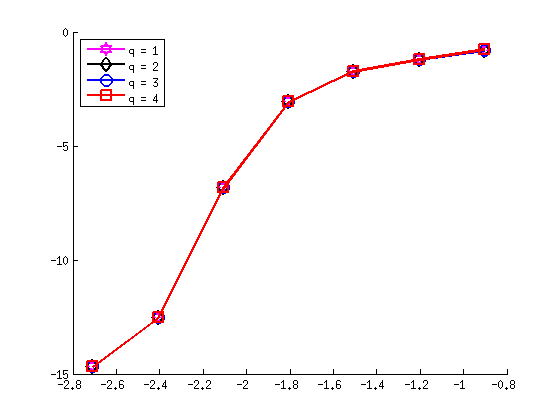}}
	\end{subfigure}

	\caption{Plots showing $\log_{10}(\varepsilon_{\infty})$ against $\log_{10}(h),\, h=1/n$ for the kernels $g(x) = |x|^\gamma,\, \gamma=-1/2$ and $u(x)=e^{-(\frac{x-0.5}{\sigma})^2}$ with $\sigma = 0.01$. 
	}
	\label{fig:-error5}
\end{figure}

\begin{figure}[t]
	\centering
	\begin{subfigure}[$q=1$ ]
		{\includegraphics[width=0.45\textwidth,trim={1.5cm 0.8cm 1.4cm 0.8cm},clip]{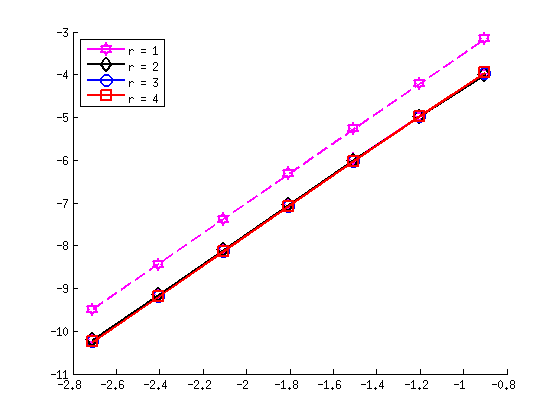}}
	\end{subfigure}
	~
	\begin{subfigure}[$r=3$ ]
		{\includegraphics[width=0.45\textwidth,trim={1.5cm 0.8cm 1.4cm 0.8cm},clip]{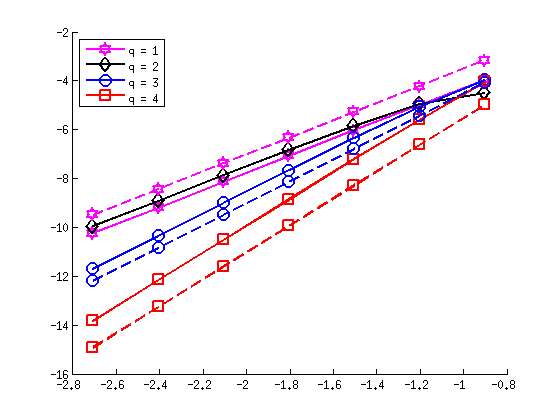}}
	\end{subfigure}

	\caption{Plots showing $\log_{10}(\varepsilon_{\infty})$ against $\log_{10}(h),\, h=1/n$ for the kernels $g(x) = |x|^{-1/2}$ and $u(x)= (x(1-x))^3 $. In each subfigure, dashed line with the slope as predicted by \Cref{th: convergence} is shown for reference.  
	}
	\label{fig:-error6}
\end{figure}

\section{summary and conclusions}  
In this paper, we have proposed a simple, efficient and high-order numerical integration scheme for the evaluation of the integral of the form
\begin{equation*}
(Au)(x) = \displaystyle \int_{0}^{1}g(x-y)u(y)dy
\end{equation*}
with a weekly singular kernel $g$, either of the form $g(x) = |x|^{\gamma}, \gamma \in (-1,\infty)$ or of the form $g(x) = \log |x|$ utilizing the Fourier extension of $u(x)$. The preeminent motivation for this Fourier extension based integration scheme is to compute integral operator $A$ with high-order accuracy in $O(n\log n)$ operations. 
This text contains a full error analysis of the numerical integration scheme as well as a range of numerical experiments. The results show a clear agreement between the thoretical rates and the observed ones. As a special case, we also discuss our numerical integration scheme when the integrand has its support contained in the domain of integration and observe that the scheme simplifies even further from the implementation point of view while maintaining the computational accuracy and efficiency.

\end{document}